\documentclass[11pt]{article}

\let\oldmarginpar\marginpar
\renewcommand\marginpar[1]{\-\oldmarginpar[\raggedleft\footnotesize #1]%
{\raggedright\footnotesize #1}}

\usepackage{mathptmx}
\usepackage{amsmath}
\usepackage{amscd}
\usepackage{amssymb}
\usepackage{amsthm}
\usepackage{xspace}
\usepackage[all,tips]{xy}
\usepackage[dvips]{graphicx}
\usepackage{verbatim}
\usepackage{syntonly}
\usepackage{hyperref}
\usepackage{arydshln}
\usepackage{lineno,color}
\usepackage{faktor}

\theoremstyle{plain}
\newtheorem{thm}{Theorem}[section]
\newtheorem{cor}[thm]{Corollary}
\newtheorem{prop}[thm]{Proposition}
\newtheorem{lemma}[thm]{Lemma}
\newtheorem{ques}{Question}

\newcommand{\suchthat}{\;\ifnum\currentgrouptype=16 \middle\fi|\;}

\theoremstyle{definition}

\newtheorem{defn}[thm]{Definition}

\newtheorem{Ex}[thm]{Example}

\DeclareMathOperator{\lcm}{lcm}

\DeclareMathOperator{\D}{D}

\DeclareMathOperator{\Farb}{F}

\DeclareMathOperator{\Sub}{Sub}\DeclareMathOperator{\Ord}{Ord}
\DeclareMathOperator{\rank}{rank}
\newcommand{\bdef}{\overset{\text{def}}{=}}

\newcommand{\ga}{\gamma}

\newcommand{\la}{\lambda}

\newcommand{\innp}[1]{\left< #1 \right>}

\newcommand{\set}[1]{\left\{#1\right\}}

\newcommand{\pr}[1]{\left( #1 \right) }

\newcommand{\N}{\ensuremath{\mathbb{N}}}

\newcommand{\Z}{\ensuremath{\mathbb{Z}}}

\newcommand{\factor}[2]{\left.\raisebox{.2em}{$#1$} \! \! \middle/ \! \! \raisebox{-.2em}{$#2$}\right.}
\newcommand{\map}[3]{#1 : #2 \rightarrow #3}

\newcommand{\nsub}{\trianglelefteq}
\usepackage{fancyhdr}

\makeatletter
\newtheorem*{rep@theorem}{\rep@title}
\newcommand{\newreptheorem}[2]{%
	\newenvironment{rep#1}[1]{%
		\def\rep@title{#2 \ref{##1}}%
		\begin{rep@theorem}}%
		{\end{rep@theorem}}}
\makeatother

\newreptheorem{thm}{Theorem}

\fancypagestyle{plain}

\begin{document}
\title{\textbf{Effective Subgroup Separability of \\ Finitely Generated Nilpotent Groups}}
\author{Jonas Der\'e and Mark Pengitore\thanks{The first author was supported by a postdoctoral fellowship of the Research Foundation - Flanders (FWO)}}
\maketitle

\begin{abstract}
	This paper studies effective separability for subgroups of finitely generated nilpotent groups and more broadly effective subgroup separability of finitely generated nilpotent groups. We provide upper and lower bounds that are polynomial with respect to the logarithm of the word length for infinite index subgroups of nilpotent groups. In the case of normal subgroups, we provide an exact computation generalizing work of the second author. We introduce a function that quantifies subgroup separability, and we provide polynomial upper and lower bounds. We finish by demonstrating that our results extend to virtually nilpotent groups and stating some open questions.
\end{abstract}

\section{Introduction}
Let $G$ be a finitely generated group with a subgroup $H$. We say that $H$ is a \emph{separable subgroup} if for each $g \in G \setminus H$ there exists a group morphism to a finite group $\map{\pi}{G}{Q}$ such that $\pi(g) \notin \pi(H)$. If the trivial subgroup is separable, we say $G$ is \emph{residually finite}. The group $G$ is called \emph{subgroup separable}, also known in the literature as locally extended residually finite (LERF), if every finitely generated subgroup of $G$ is separable. Subgroup separability is thus a natural generalization of residual finiteness.

The study of subgroup separability in the literature has been to understand which groups satisfy these properties. For instance, closed surface groups, free groups, fundamental groups of geometric $3$-manifolds, finitely generated nilpotent groups, and polycyclic groups have all been shown to be subgroup separable and subsequently, residually finite in \cite{Agol_virtual_haken,Hall_coset,Mal01,scott_surface_subgroups,Segal_book_polycyclic,Wise}. Recently, there is a lot of interesting in making effective various separability properties such as residual finiteness and subgroup separability. For a finitely generated group $G$ with a finite generating subset $S$ and a finitely generated subgroup $H \leq G$, we introduce a function $\Farb_{G,H,S}(n)$ on the natural numbers that quantifies the separability of $H$ in $G$. In particular, the value $\Farb_{G,H,S}(n)$ on a natural number $n$ is such that every element in the complement of $H$ of word length at most $n$ can be distinguished in a finite quotient of order at most $ \Farb_{G,H,S}(n)$. One can see that $\Farb_{G,H,S}(n)$ is a generalization of the function introduced for residually finiteness in \cite{Bou_Rabee10} to arbitrary finitely generated subgroups. By generalizing word length of group elements to finitely generated 
subgroups, we are also able to quantify subgroup separability. To be specific, the function $\Sub_{G,S}(n)$ on the value $n$ is the minimal value such that every finitely generated subgroup $H$ can be separated from an element $g$ in the complement of $H$ in a finite quotient of order at most $\Sub_{G,S}(n)$ as one varies over subgroups $H$ and elements $g$ satisfying $\|H\|_S, \|g\|_S \leq n$. 

Previous work on the function $\Farb_{G,H,S}(n)$ has fallen into two different contexts. When $H = \set{1}$, many papers have been written that explore the asymptotic behavior of $\Farb_{G, \set{1},S}(n)$. See \cite{LLM} and the references therein for a more complete account of the literature. When $H$ is a nontrivial, proper, finitely generated subgroup, $\Farb_{G,H,S}(n)$ has been studied when $G$ is a closed surface group, a free group, or a virtually compact special hyperbolic group. The papers \cite{zariski_closure_subgroup_separability,Patel} imply that if $G$ is a free group or a surface group and $H$ is any finitely subgroup, then there exists a $d \in \N$ such that $\Farb_{G,H,S}(n) \preceq n^d$. Likewise, \cite{Patel3} implies that if $G$ is a virtually compact special hyperbolic group and $H \leq G$ is a $K$-quasiconvex subgroup, then $\Farb_{G,H,S}(n)$ is bounded by a function which is polynomial in $n$ and exponential in $K$. In the above cases, no lower asymptotic bound was provided; moreover, there has been no prior work on the function $\Sub_{G,S}(n)$. However, providing a bound for  $\Sub_{G,S}(n)$ is similar to the result of \cite{Patel3} except the complexity of the subgroups is given by the subgroup norm instead of the quasiconvexity constant. It is also relevant for studying twisted conjugacy separability in finitely generated groups, for example see \cite[Question 1]{dere_pengitore_17}.

This article provides the first asymptotic bounds for $\Sub_{G,S}(n)$ which, in turn, provides an universal asymptotic upper bound for $\Farb_{G,H,S}(n)$ independent of the subgroup $H$. We are also the first to provide precise asymptotic bounds for $\Farb_{G,H,S}(n)$ when $H$ is a nontrivial normal subgroup and provides asymptotic lower bounds for general subgroups. These results are all in the context of nilpotent groups which generalizes the work in \cite{corrigendum_pengitore}.

To state our results, we require some notation. For two non-decreasing functions $\map{f,g}{\N}{\N}$, we write $f(n) \preceq g(n)$ if there exists a $C \in \N$ such that $f(n) \leq C \: g(C n)$ for all $n$. We write $f \approx g$ when both $f \preceq g$ and $g \preceq f$. For the next theorem, we observe that if $N$ is a finitely generated group and $H \neq N$ is a finite index subgroup, then we may pass to the normal core of $H$, denoted $H_0$, to conclude that $\Farb_{N,H,S}(n) \leq \left \vert \faktor{N}{H_0} \right\vert $. Thus, we assume that $H$ is an infinite index subgroup of $N$ to formulate the result.
\begin{thm}\label{effective_separation_subgroup_thm}
	Let $N$ be a torsion-free, finitely generated nilpotent group with a finite generating subset $S$, and suppose that $H \leq N$ is a subgroup of infinite index. Then for $h(N)$ the Hirsch length of $N$, it holds that $$\log(n) \preceq \Farb_{N,H,S}(n) \preceq \left(\log(n)\right)^{h(N)}.$$ Moreover, if $H \nsub N$, there exists a smaller constant $\psi(N,H) \in \N$ such that $$\log(n) \preceq \Farb_{N,H,S}(n) \preceq \pr{\log(n)}^{\psi(N,H)}.$$
\end{thm}

The constant $\psi(N,H)$ is given in Definition \ref{constants_defn} and can be computed given groups $N$ and $H$. The proof of the first statement follows from induction on the Hirsch length and making effective separating a nontrivial central element from an arbitrary subgroup, and the proof of the second statement is elementary and follows from a close inspection of how the asymptotic behavior of $\Farb_{N,H,S}(N)$ relates to the asymptotic behavior of $\Farb_{\faktor{N}{H},\set{1},\bar{S}}(n)$ and from \cite{corrigendum_pengitore}.

In the context of abelian groups, the explicit form of the constant $\psi(N,H)$ gives the following consequence.
\begin{cor}
	\label{abeliancor}
	Let $A$ be a torsion free, finitely generated abelian group with a finite generating subset $S$. If $B \leq A$ is a subgroup of infinite index, then $\Farb_{A,B,S}(n) \approx \log(n)$.
\end{cor}
This shows that for a given nilpotent group $N$, the lower bound $\log(n)$ is the best we can do since $H = [N,N]$ will lead to $\Farb_{N,H,S}(N) \approx \log(n)$. 

For general nilpotent groups $N$, we establish asymptotic upper and lower bounds for $\Sub_{N,S}(n)$.

\begin{thm}\label{effective_subgroup_separability_thm}
	Let $N$ be a torsion-free, finitely generated nilpotent group with a finite generating subset $S$. There exists a $k \in \N$ such that $$n \preceq \Sub_{N,S}(n) \preceq n^{k}.$$
\end{thm}
The proof of Theorem \ref{effective_subgroup_separability_thm} is done in two steps. For the upper bound, we introduce some tools that bound the norm of the intersection of a subgroup $H \leq N$ with terms of a central series by the norm of $H$. To establish the lower bound, we construct a sequence of subgroups $H_i$ and elements $g_i$ such that $g_i \notin H_i$ and the order of the minimal finite quotient that distinguishes $g_i$ and $H_i$ is bounded below by $\Vert H_i \Vert_S$. Note that we could express the constant $k$ in terms of the Hirsch length of $N$, but we avoid this since the constructed bound in Theorem \ref{effective_subgroup_separability_thm} is far from sharp.

Finally, we show that the separability functions above behave well under finite extensions. As an application, we generalize the polynomial upper bounds and lower bounds to virtually nilpotent groups. We state some open related questions in the last section.

\section{Background and notation}

Let $G$ be a finitely generated group. We denote the identity element of $G$ as $1$, and when $G$ is an abelian group, we use  additive notation and take $0$ as the identity. We denote the commutator of $x,y \in G$ as $[x,y] = x \: y \: x^{-1} \: y^{-1}$. For two subsets $A,B \leq G$, we denote $[A,B]$ to be the subgroup of $G$ generated by elements of the form $[a,b]$ for $a \in A$ and $b \in B$. For  $H \nsub G$, we set $\map{\pi_H}{G}{\faktor{G}{H}}$ to be the natural projection. We define $G^m$ to be the subgroup generated by $m$-th powers of elements in $G$ with associated projection $\pi_{m}.$ We denote the center of the group $G$ as $Z(G)$, and the order of a finite group is denote as $|G|$. The order of an element $g \in G$ is denoted as $\Ord_G(g)$. For a natural $m$ and a prime $p$, we denote $\nu_p(m)$ as the largest power of $p$ such that $p^{\nu_p(m)}$ divides $m$. \label{mappik}

For any subset $X \subseteq G$, we let $\langle X \rangle$ be the subgroup generated by the set $X$. When $G$ is finitely generated with a finite generating subset $S$, we write $\|x\|_S$ as the word length of $x$ with respect to $S$. We can also define the norm of subgroups of $G$.

\begin{defn}
	Let $G$ be a finitely generated group with a finite generating subset $S$. For any finite subset $X \subseteq G$, define $\|X\|_S \bdef \max \left\{\|x\|_S \: | \: x \in X \right\}$. For a finitely generated subgroup $H \leq G$, let
	$$
	\|H\|_S \bdef \min \left\{ \|X\|_S \: | \: X \text{ is a finite generating subset for } H \right\}.
	$$
\end{defn}
If $S_1$ and $S_2$ are any two finite generating subsets of a group $G$, then there exists a constant $C > 0$ such that $\| g \|_{S_2} \leq C \| g \|_{S_1}$ for all $g \in G$. It then follows that $\|H\|_{S_2} \leq C \|H\|_{S_1}$, since if $\{h_i\}_{i=1}^k$ is a set of generators for $H$ such that $\|h_i\|_{S_1} \leq \|H\|_{S_1}$ for all $i$, then $$\|h_i\|_{S_2} \leq C \|h_i\|_{S_1} \leq C \|H\|_{S_1}.$$ Since these elements generate $H$, we have the following lemma.
\begin{lemma}\label{subgroup_norm_lemma}
	Let $G$ be a finitely generated group with a finite generating subsets $S_1$ and $S_2$. There exists a constant $C > 0$ such that if $H \leq G$ is a finite generated subgroup, then $$C^{-1}\|H\|_{S_1} \leq \|H\|_{S_2} \leq C \: \|H\|_{S_1}.$$
\end{lemma}

\paragraph{Nilpotent groups}
$ $\newline
The groups we work with in this paper are nilpotent groups, and we recall their definition and basic properties. See \cite{Hall_notes,Segal_book_polycyclic} for a more complete account of the theory of nilpotent groups.

A \textit{central series} for a group $N$ is a sequence of subgroups $N = N_0 \geq N_1 \geq \ldots \geq N_k = 1$ such that $[N,N_{i}] \leq N_{i+1}$ for all $i$. The $i$-th term of the \emph{lower central series} is defined by $\ga_1(N) = N$ and inductively by $\ga_i(N) = [N,\ga_{i-1}(N)]$. We say that a group $N$ is \emph{nilpotent of step size $c$} if $c$ is the minimal natural number such that $\ga_{c+1}(N) = 1$.  When $N$ is nilpotent, we denote its nilpotency class of $N$ as $c(N)$, and if we don't specify the step size, we say that $N$ is a nilpotent group. When $N$ is nilpotent, the lower central series forms, as the name says, a central series for the group $N$. We define the \emph{Hirsch length of $N$} as
	$$
	h(N) \bdef \sum_{i=1}^{c(N)}\rank_\Z \pr{\factor{\ga_i(N)}{\ga_{i+1}(N)}}.
	$$
	If $N$ is a torsion-free, finitely generated nilpotent group, we say $N$ is a $\mathcal{F}$-group.

\begin{defn}
	Let $N$ be a $\mathcal{F}$-group. We call a central series $\set{N_i}_{i=0}^{h(N)}$ \emph{maximal} if $\faktor{N_i}{N_{i+1}}\cong \Z$ for all $0 \leq i \leq h(N)-1$. 
\end{defn}
Maximal series always exist for $\mathcal{F}$-groups; however, they are not unique. Their existence is guaranteed by \cite[Lemma 8.23(c)]{Eick_Holt_obrien}.

\begin{defn}
Let $N$ be a $\mathcal{F}$-group, and let $H \leq N$ be a subgroup. We define the \emph{isolator of $H$ in $N$}, denoted $\sqrt[N]{H}$, as the set
	$$
	\sqrt[N]{H} \bdef \set{ \left. x \in N \: \right| \: \text{ there exists a } k \in \N \text{ such that } x^k \in H}.
	$$
\end{defn}
From \cite{Segal_book_polycyclic}, it follows that if $N$ is a finitely generated nilpotent group, then $\sqrt[N]{H}$ is a subgroup for all $H \leq N$, which moreover satisfies $|\sqrt[N]{H}:H| < \infty$. 

\begin{defn}
	Let $N$ be a $\mathcal{F}$-group, and let $z \in N$ be a primitive central element of $N$. An \emph{admissible quotient} of $N$ associated to $z$ is a quotient $\faktor{N}{H}$ such that $\faktor{N}{H}$ is a $\mathcal{F}$-group where $Z\left( \faktor{N}{H} \right) = \innp{\pi_H(z)}$.
\end{defn}

For any primitive central element, the existence of an associated one dimensional central quotient of $N$ is guaranteed by \cite[Proposition 3.1]{Pengitore_1}. 

\begin{defn}\label{constants_defn}
	Let $N$ be a $\mathcal{F}$-group. We define $\Phi(N)$ to be the smallest integer such that for every primitive element $z \in Z(N)$, there exists an admissible quotient $\faktor{N}{H}$ associated to $z$ such that $h\left(\faktor{N}{H}\right) \leq \Phi(N)$. For a normal subgroup $H \nsub N$ of a $\mathcal{F}$-group, we define $\psi(N,H) = \Phi\left(\faktor{N}{ \sqrt[N]{H}}\right)$. 
\end{defn}
For a torsion-free, finitely generated abelian group $A$ and any primitive element $z$, one can see if $\faktor{A}{B}$ is an admissible quotient of $A$ associated to $z$, then $\faktor{A}{B} \cong \Z$. In particular, $\Phi(A) = 1$.

\paragraph{Effective separability}$ $\newline
Let $G$ be a group with a proper subgroup $H$. Following \cite{Bou_Rabee10}, we defined the relative depth function $
\map{\D_G(H,\cdot)}{G \setminus H}{\mathbb{N} \cup \set{\infty}}$ of $H$ in $G$ as
$$
\D_G(H,g) \bdef \min \set{|Q| \: \: \suchthat  \: \text{ there exists a } \map{\pi}{G}{Q} \text{ such that } |Q| < \infty \text{ and } \pi(g) \notin \pi(H)}
$$
with the understanding that $\D_G(H,g) = \infty$ if no such $Q$ exists.
\begin{defn}
	We say that a finitely generated subgroup $H \leq G$ is \emph{separable} if $\D_G(H,g) < \infty$ for all $g \in G \setminus H$. We say that a finite group $Q$ \emph{separates} $H$ and $g$ if there exists a surjective group morphism $\map{\pi}{G}{Q}$ such that $\pi(g) \notin \pi(H)$. We refer to the process of finding finite groups that separate $H$ from elements of $G \setminus H$ as \emph{separating $H$ in $G$.}
\end{defn}

With the above definition in mind, we can define the residual property of interest for this article.
\begin{defn}
	We say a finitely generated group $G$ is \emph{subgroup separable} if every finitely generated group is separable. 
\end{defn}

Now assume that $G$ is finitely generated by a finite generating subset $S$, and let $H \leq G$ be a proper, finitely generated, separable subgroup. To quantify the complexity of separating $H$ in $G$, we introduce the function $\map{\Farb_{G,H,S}}{\N}{\N}$ given by
$$
\Farb_{G,H,S}(n) \bdef \max \set{\D_G(H,g) \: | \: \|g\|_S \leq n \text{ and } g \in G \setminus H}.
$$

\begin{lemma}
	If $S_1$ and $S_2$ are two finite generating subsets of $G$ and $H$ is a finitely generated separable subgroup, then $\Farb_{G,H,S_1}(n) \approx \Farb_{G,H,S_2}(n)$.
\end{lemma}
The proof follows from standard arguments for the word norm. 

It is well known that every subgroup of a finitely generated nilpotent group is finitely generated. Thus, whenever we reference the function $\Farb_{N,H,S}(n)$ for a $\mathcal{F}$-group $N$ and a finitely generated subgroup $H \leq N$, we will simply say that $H$ is a subgroup.

The following function allows us to quantify the complexity of subgroup separability for any finitely generated subgroup separable group $G$ with a finite generating subset $S$.
\begin{defn}
	Let $G$ be a finitely generated subgroup separable group with a finite generating subset $S$. Define $\map{\Sub_{G,S}}{\N}{\N}$ as
	$$
	\Sub_{G,S}(n) \bdef \max\{\D_G(H,g) | \: H \leq G \text{ finitely generated, } g \in G \setminus H, \text{ and } \|H\|_S, \|g\|_S \leq n\}.
	$$
\end{defn}

\begin{lemma}
	If $S_1$ and $S_2$ are two finite generating subsets of $G$, then $\Sub_{G,S_1}(n) \approx \Sub_{G,S_2}(n)$.
\end{lemma}
As before, the proof is similar to \cite[Lemma 1.1]{Bou_Rabee10}, but we additionally appeal to Lemma \ref{subgroup_norm_lemma}.

We note that for a $\mathcal{F}$-group $N$, we may define $\Sub_{N,S}(n)$ as
$$
\Sub_{N,S}(n) = \max \set{\D_G(H,g) \: | \: H \leq N, g \in N \setminus H \text{ and } \|H\|_S,\|g\|_S \leq n}.
$$

\section{Intersections of Subgroups with Normal Series and Applications}
Let $N$ be a $\mathcal{F}$-group with a finite generating subset $S$ and a maximal central series $\set{N_i}_{i=0}^{h(N)}$. For any subgroup $H \leq N$, we seek to estimate $\|H \cap N_i\|_S$ in terms of $\|H\|_S$. This result will be essential for the inductive step in the proof of Theorem \ref{effective_subgroup_separability_thm}.

The following theorem is an effective version of Bezout's Lemma for $\Z$. This result and its proof are originally from \cite{optimal_Bezout_coefficients_bound}.

\begin{thm}
	\label{effBez}
	Let $a_1, \ldots, a_n \in \Z$ be any number of integers. There exist $x_1, \ldots, x_n \in \Z$ with $\vert x_i \vert \leq \frac{\max \{ \vert a_i \vert \}}{2}$ such that $$\sum_{i=1}^n x_i \: a_i = \gcd(a_1,\ldots, a_n).$$
\end{thm}

Note that the subgroup $H = \langle a_1, \ldots, a_n \rangle \le \Z$ satisfies $H = \gcd(a_1, \ldots, a_n) \: \Z$. In general, we don't have a bound on the number of generators of a subgroup. Therefore, we need the following lemma for subgroups of $\Z$ which measures the length of the standard generating subset of $\Z$ with respect to the given finite generating subset.

\begin{lemma}
	\label{effBez2}
	Let $S$ be any finite generating subset for $d \Z$ and assume that $\vert s \vert \leq n$ for all $s \in S$. Then $$\Vert d \Vert_S \leq n^2.$$
\end{lemma}

\begin{proof}
	We may assume that $0 \notin S$; hence, there are at most $2n$ elements in $S$. Write $S = \{s_1, \ldots, s_k\}$ with $k \leq 2n$. By using Theorem \ref{effBez}, we get that $d = \sum_{i=1}^k s_i \: x_i$ with $\vert x_i \vert \leq  \frac{ n }{2}$. In particular, 
	$$\Vert d \Vert_S \leq \sum_{i=1}^k \vert x_i \vert \leq k \frac{n}{2} \leq n^2. \qedhere$$ \end{proof}

\noindent From the proof, it follows that this result can easily be improved to a bound of $n \: \log(n)$; however, this does not improve other results so we only formulate it as $n^2$.

The following lemma gives a bound for the norm of $\ga_2(N)$ for a $\mathcal{F}$-group $N$ dependent only on the nilpotency class of $N$.
\begin{lemma}
\label{normgamma2}
For every $c \in \N$, there exists a constant $K_c$ such that for all nilpotent groups $N$ of nilpotency class at most $c$ and every finite generating subset $S$ for $N$, it holds that $\Vert \gamma_2(N) \Vert_S \leq K_c.$
\end{lemma}
\begin{proof}
From \cite[Lemma 1.7]{Hall_notes}, it follows for any finite generating subset $S$ of $N$, the subgroup $\gamma_2(N)$ is generated by the elements of the form $[s_1,[s_2, \cdots,[s_{k-1},s_k], \cdots]]$ with $s_i \in S$ and $k \leq c$. The lemma is now immediate.
\end{proof}

Suppose that $N$ is a $\mathcal{F}$-group with a finite generating subset $S$ and a maximal central series $\set{N_i}_{i=0}^{h(N)}$, and let $H \leq N$ be a subgroup. Lemma \ref{effBez2} allows us to estimate $\|H \cap N_i\|_S$ in terms of $\|H\|_S$.

\begin{prop}\label{norm_intersect_central}
Let $N$ be a $\mathcal{F}$-group with a finite generating subset $S$. Let $\set{N_i}_{i=0}^{h(N)}$ be a maximal normal series. There exists a natural number $k$ and a constant $C>0$ such that for every $i \geq 0$ and every subgroup $H \leq N$, it holds that $$
\|H \cap N_i\|_S \leq C (\|H\|_S)^k.$$
\end{prop}
\begin{proof}
	We prove the proposition by induction on Hirsch length, and note that the base case is clear by definition. Thus, we may assume that $h(N) > 1$. For notational simplicity, we let $\pi = \pi_{N_1}$.
	
	The statement is invariant under change of finite generating subset, so we may assume that there exists a $s_0 \in S$ such that $\faktor{N}{N_1} \cong \innp{\pi(s_0)}$ and $S = \left\{ s_0\right\} \cup S^\prime$ where $S^\prime$ is a finite generating subset for $N_1$. By \cite[3.B2]{Gromov}, there exists a $k_1 \in \N$ and a constant $C_1 > 0$ such that $\|x\|_{S} \leq C_1 \pr{\|x\|_{S'}}^{k_1}$ and $\|x\|_{S'} \leq C_1 \pr{\|x\|_{S}}^{k_1}$ for all $x \in N_1$. Take any subgroup $H \le N$ and any finite generating subset $T_1$ for $H$ with $\Vert t \Vert_S \leq \Vert H \Vert_S$ for all $t \in T_1$.
	
First suppose that $\pi(H) = 0$ or equivalently $H \le N_1$. Note that $N_1 \ge N_2 \ge \ldots \ge N_{h(N)}$ is a maximal central series for $N_1$. From the induction hypothesis, the proposition holds for the group $N_1$ with some constant $C_2 > 0$ and $k_2 \in \N$. So for all $i \geq 1$ we get
	$$\|H \cap N_i\|_S \leq C_1 \pr{ \|H \cap N_i\|_{S'}}^{k_1} \leq C_1 \: C_2^{k_1} \pr{\|H\|_{S'}}^{k_1 \: k_2} \leq C_1 \: C_1^{k_1 \: k_2} C_2^{k_1} \pr{\|H\|_S}^{k_1^2 \: k_2}. $$ On the other hand, for $i = 0$ it holds that $H \cap N_0 = H$, and thus, the statement of the proposition is evidently true. Hence, the proposition holds for subgroups $H \le N_1$.

	Thus, we may assume that $\pi(H) \neq 0$. By Lemma \ref{effBez2} with the generating subset $\pi(T_1)$ for $\pi(H)$, we find an element $t_0 \in H$ such that $\pi(t_0) \neq 0$ generates $\pi(H)$ with $\Vert t_0 \Vert_S \leq \Vert T_1 \Vert_S^2 \leq \pr{\|H\|_S}^{2}$. We now construct a finite generating subset for $H \cap N_1$. For each generator $t \in T_1$, let $x_t \in N_1$ be given by $x_t = t \: t_0^{- \frac{\pi(t)}{\pi(t_0)}}$. Note that these elements indeed lie in $N_1$, since $$\pi(x_t) = \pi(t) - \pi(t_0) \frac{\pi(t)}{\pi(t_0)} = 0.$$ 
	
Take a generating subset $T_2^\prime$ for $\gamma_2(H)$ such that $\Vert T_2^\prime \Vert_{T_1} \leq K_c$, which exists by Lemma \ref{normgamma2}.	We claim the set $T_2 = T_2^\prime \cup \set{x_t \: | \: t \in T_1}$ is a finite generating subset for $H \cap N_1$.  We only need to demonstrate that the image of the set $\set{x_t \: | \: t \in T}$ in $\faktor{H}{[H,H]}$ is a finite generating subset for $\faktor{H \cap N_1}{[H,H]}$, where $[H,H] \le N_1$ by definition. Let $T_1 = \left\{ t_1, \cdots, t_l \right\}$ and take $h \in H \cap N_1$. We may write $$h \equiv \: \prod_{i=1}^\ell t_{i}^{m_i}  \mod  [H,H]$$ for some integers $m_i$, and by construction, $\prod_{i=1}^\ell t_{i}^{m_i} \equiv h \equiv 0 \mod H \cap N_1$. Since $t_{i} \equiv t_0^{\pi(t_{i}) / \pi(t_0)} \mod N_1$, we have
	$$
	h \equiv \prod_{i=1}^\ell t_{i}^{m_i} \equiv t_0^{\sum_{i=1}^{\ell} m_i \: d_i \: \pi(t_0)} \mod N_1,
	$$
	where $d_i = \frac{\pi(t_{i})}{\pi(t_0)}$.
	In particular, $\sum_{i=1}^{\ell} m_i \: d_i = 0$. We may write
	$$
	\prod_{i=1}^\ell x_{t_{i}}^{m_i} \equiv \prod_{i=1}^{\ell}t_{i}^{{m_i}} t_0^{-{m_i \: d_i}} \equiv t_0^{^{-\sum_{i=1}^{\ell} m_i \: d_i}} \: \prod_{i=1}^\ell \: t_{i}^{m_i} \equiv h \mod [H,H].
	$$
	Thus, $$
	\prod_{i=1}^\ell x_{t_{i}}^{m_i} 
	\equiv h \mod [H,H].$$ Hence, $T_2$ is a finite generating subset for $H \cap N_1$.
	
	We now find a bound for $\|H \cap N_1\|_S$ by providing a bound for $\|T_2\|_S$ in terms of $\|H\|_S$. We first note that if $t \in T_2^\prime$, then $\|t\|_S \leq K_{c}  \|H\|_S$. The norm of each $x_t$ satisfies 
	$$
	\Vert x_t \Vert_S \leq |\pi(t)|  \| t_0 \|_S + \| t \|_S \leq \Vert H \Vert_S ^3 +  \Vert H \Vert_S \leq 2 \Vert H \Vert_S^3.
	$$ 
In particular, $\|H \cap N_1\|_S \leq C_3 \pr{\|H\|_S}^{3}$ for some constant $C_3 > 0$. Since we reduced the general case to the situation where $H \le N_1$, we are finished.
\end{proof}

We finish this section with some applications of Proposition \ref{norm_intersect_central}. Let $N$ be a $\mathcal{F}$-group with a maximal central series $\set{N_i}_{i=0}^{h(N)}$. Suppose $H \leq N$ is a subgroup with an element $g \in N \setminus H$. When $\pi_{N_{h(N)}}(g) \in \pi_{N_{h(N)}}(H)$, the following proposition constructs an element $z \in N_{h(N)} \setminus H$ such that $g = z \: h$ for some $h \in H$ and where there exists a bound on the word length of $z$ with respect to the word length of $g$. Once a bound for the word length of $z$ has been found, we may find a finite group that separates $g$ and $H$ by finding a finite group that separates $z$ and $H$ as will be seen at the end of this section.
\begin{prop}\label{length_of_g_mod_H}
	Let $N$ be a $\mathcal{F}$-group with a finite generating subset $S$, and let $\set{N_i}_{i=0}^{h(N)}$ be a maximal central series. There exists some constant $C > 0$ and $k \in \N$ such that for every subgroup $H \leq N$ and $g \in N \setminus H$ with $\pi_{N_{h(N)}}(g) \in \pi_{N_{h(N)}}(H)$, there exists a $z \in N_{h(N)} \setminus H$ such that $g \: z \in H$ and $$\|z\|_S \leq C \: \left( \max\{\|H\|_S,\|g\|_S\}\right)^{k}.$$
\end{prop}
\begin{proof}
We proceed by induction on the Hirsch length, and observe that the base case is clear. We may assume that $S = S' \cup \set{s_0}$ where $s_0 \: N_1$ generates $\faktor{N}{N_{1}}$ and $S'$ generates $N_1$. Assume $h(N) > 1$, and let $T$ be a finite generating subset for $H$ such that $\|t\|_S \leq \|H\|_S$ for all $t \in T$. 

If $g \in N_1$, then the inductive hypothesis implies that there exists a $z \in N_{h(N)} \setminus H \cap N_1$ such that $g \: z \in H$ and $$\|z\|_{S'} \leq C_1 \pr{\max \{\|H \cap N_1\|_{S'},\|g\|_{S'}}^{k_1}$$ where $C_1 > 0$ is some constant and $k_1 \in \N$. \cite[3.B2]{Gromov} implies that there exists a constant $C_2 > 0$ and $k_2 \in \N$ such that $\|g\|_{S'} \leq C_2 \pr{\|g\|_{S}}^{k_2}$ (hence also $\|H\|_{S'} \leq C_2 \pr{\|H\|_{S}}^{k_2}$) and $\|g\|_{S} \leq C_2 \pr{\|g\|_{S'}}^{k_2}$ for all $g \in N_1$. Thus,  $$\|z\|_{S} \leq   C_1^{k_2} \: C_2^{k_1 \: k_2} C_2 \: \left(\max \{\|H \cap N_1\|_S,\|g\|_S\}\right)^{k_1 \: k_2^2}$$ just as in the proof of the previous proposition. Proposition \ref{norm_intersect_central} implies that there exists a constant $C_3 > 0$ and $k_3 \in \N$ such that $\|H \cap N_1\|_S \leq C_3 \: \pr{\|H\|}^{k_3}.$   Hence,
$$
\|z\|_{S} \leq  C_1^{k_2} \: C_2^{k_1 \: k_2} C_2  \: C_3^{k_1 \: k_2^2} \: \pr{\max\{\|H\|_S,\|g\|_S\}}^{k_1 \: k_2^2 \: k_3}.$$

Now suppose $g \notin N_{1}$. Since $\pi_{N_{h(N)}}(g) \in \pi_{N_{h(N)}}(H)$, we have $\pi_{N_1}(g) \in \pi_{N_1}(H)$. There exists a $d_1 \in \Z \setminus \set{0}$ such that $|d_1| \leq \Vert g \Vert_S$ and $g \equiv s_0^{d_1} \mod N_1$. Lemma \ref{effBez2} implies that there exists a $t_0 \in H$ such that $\langle\pi_{N_1}(t_0)\rangle \cong \pi_{N_1}(H)$ and $\Vert t_0 \Vert_S \leq (\|H\|_S)^2$. Given that $\pi_{N_1}(s_0^{d_1}) \in \pi_{N_1}(H)$, there exists a $k \in \Z$ such that $\pi_{N_1}(t_0^k) = \pi_{N_1}(g) = \pi_{N_1}(s_0^{d_1})$, and thus, $\vert k\vert \leq \vert d_1 \vert \leq n$. Letting $h = t_0^k$, we have $g \: h^{-1} \in N_1 \setminus H$, and since $$\|h^{-1}\|_S = \Vert h \Vert_S \leq \vert k \vert \Vert t_0 \Vert_S \leq n \Vert H \Vert_S^2 \leq \Vert g \Vert_S\Vert H \Vert_S^2,$$ we may proceed as in the previous case to find the general result.
\end{proof}

This next lemma and its proof can be found in \cite[Lemma 3.10]{dere_pengitore_17}. Moreover, this lemma will be useful in separating a central element from a subgroup in a $\mathcal{F}$-group.
\begin{lemma}\label{power_lemma}
Let $N$ be a $\mathcal{F}$-group of nilpotency class $c$, and let $p$ be a prime. There exists an integer $k(p,c) \geq 0$ such that if $x \in N^{p^{k + k(p,c)}}$, then there exists a $y \in N$ such that $x = y^{p^k}$. Additionally, $k(p,c)$ can be chosen so that $p^{k(p,c)} \leq c!$ for all primes $p$.
\end{lemma}

This last proposition gives the effective behavior of separating a central element $z$ from any subgroup. See \cite[Proposition 5.3]{dere_pengitore_17} for a similar statement in the context of central subgroups. In this proposition, the map $\pi_{p^k}: N \to \faktor{N}{N^{p^k}}$ is the natural projection map, where $N^{p^k}$ is the (normal) subgroup generated by all the $p^k$-powers of elements in $N$, see page \pageref{mappik}.

\begin{prop}\label{central_element_sep_subgroup}
	Let $N$ be a $\mathcal{F}$-group with a finite generating subset $S$. There exists a constant $C > 0$ and a $m \in \N$ such that for all subgroups $H \leq N$ and every $x \in Z(N) \setminus H$, there exists a prime power $p^k$ such that $\pi_{p^k}(x) \notin \pi_{p^k}(H)$ where
	$$
	p^k \leq C (\|H\|_S)^m \: \log(C \|x\|_S).
	$$
\end{prop}
\begin{proof}
Let $H$ be a subgroup and $z \in Z(N) \setminus H$. Let $\set{N_i}_{i=0}^{h(N)}$ be maximal central series where there exits $i_0$ with $N_i \leq Z(N)$ for $i \geq i_0$, $Z(N) \leq N_i$ for $i \leq i_0$, and subsequently, $N_{i_0} = Z(N)$. Note that such a maximal central series always exists. Taking $H_i = H \cap N_i$, we will first prove the statement for the subgroup $H_{i_0}$ and then show that if the statement holds for the subgroup $H_{i+1}$ with $i + 1 \leq i_0$, then it also holds for the subgroup $H_i$. The proposition then follows by applying Proposition \ref{norm_intersect_central}.

First take $H = H_{i_0}$ which implies that $H \le Z(N)$. If $z \notin \sqrt[N]{H}$, then the statement follows from separating the image of $z$ from the identity in the group $\faktor{N}{\sqrt[N]{H}}$, see \cite{Bou_Rabee10}. Thus, we may assume that $z \in \sqrt[N]{H}$. Take a primitive element $s \in Z(N)$ such that $z = s^{d_1}$ and $H \cap \langle s \rangle = \langle s^{d_2} \rangle$ for $d_1, d_2 \in \Z$. There exists some constants $C_1 > 0, \hspace{0.5mm} m_1 \in \N$ such that $d_2 \leq C_1 \pr{\|H\|_S}^{m_1}$ from \cite[Proposition 4.2.]{dere_pengitore_17}. Since $z \notin H$, there exists a prime power such that $p^k \nmid d_1$ but $p^k \mid d_2$ (and hence $p^k \leq \vert d_2 \vert$). Considering the subgroup $N^{p^{k+ k(p,c)}}$, Lemma \ref{power_lemma} implies that $N^{p^{k + k(p,c)}} \cap Z(N) \le p^k Z(N)$. That implies that $z$ is separated from $H$ in the quotient $\faktor{N}{N^{p^{k+k(p,c)}}}$, from which the statement follows. 

Suppose now the statement holds for $H_{i+1}$ with $i+1 \leq i_0$ or equivalently $Z(N) \le N_{i+1}$, we then show it holds for the subgroup $H_i$. By assumption, $z \in N_{i+1}$. We know that there exists a constant $C_2 > 0$ and a $m_2 \in \N$ such that there exists a prime power $p^\ell$ such that $\pi_{p^{\ell}}(x) \notin \pi_{p^{\ell}}(H_{i+1})$ and where$$
	p^{\ell} \leq C_2 (\|H\|_S)^{m_2} \log(C_2 \: \|x\|_S).
	$$
	If $H_i = H_{i+1}$, then there is nothing to show. Thus, we may assume that there exists a $h_{i} \in N$ such that $\innp{h_{i},H_{i+1}} = H_{i}$. Letting $g_{i}$ be a generator for the quotient $\faktor{N_i}{N_{i+1}}$, there exists a $d \in \Z$ such that $h_{i} \equiv g_{i}^{d}  \mod N_i$. 

	If $\pi_{p^\ell}(z) \notin \pi_{p^\ell}(H_{i})$, we are done. Otherwise, we have $z \equiv h \mod N^{p^{\ell}}$ for some $h \in H_{i}$. We may write $h = h_{i}^j \: h_{i+1}$ where $h_{i+1} \in H_{i+1}$. In particular, $z \equiv g_{i+1}^{jd} \mod N_i$. Letting $$\ell_0 = k(p,c(N)) + \nu_p(d),$$ we claim that $\pi_{p^{\ell + \ell_0}}(z) \notin \pi_{p^{\ell + \ell_0}}(H_{i+1})$.
	
	 For a contradiction, assume otherwise. We have $h_{i}^{j} \: h_{i+1} \: \in z \: N^{p^{\ell + \ell_0}}$. That implies $$h_{i}^j \: N_{i+1} \in  N^{p^{\ell + \ell_0}} N_{i+1} \le \langle g_i^{p^{l+ \nu_p(d)}} \rangle N_{i+1}.$$ We note that $h_{i}^j \equiv  g_{i}^{jd} \mod N_{i+1}$, and thus, $$
	 g_{i}^{jd} \in \langle g_i^{p^{l+ \nu_p(d)}} \rangle \mod N_{i+1}.
	 $$
	 In particular, $p^{\ell} \mid j$, and hence, $h_{i}^j \in N^{p^{\ell}}$. We get that 
	$$
	 z = h_{i}^j \: h_{i+1} N^{p^{\ell}} \in H_{i+1} \: N^{p^{\ell}} 
	$$ which is a contradiction.

	To provide a bound for $p^{\ell + \ell_0}$, we note that we have a bound for $p^\ell$ by induction and that Lemma \ref{power_lemma} implies $p^{k(p,c(N))} \leq c(N)!$. Thus, we need a bound for $d$. Proposition \ref{norm_intersect_central} implies that $\|H_{i}\|_S \leq C_3 \pr{\|H\|_S}^{k_3}$ for some constant $C_3 > 0$ and integer $k_3 \in \N$. In particular, there exists a finite generating subset $\set{a_t}_t$ for $H_{i}$ such that $\|a_t\|_S \leq C_3 \pr{\|H\|_S}^{k_3}$ for all $t$. Lemma \ref{effBez2}
	implies with the generating subset $\set{\pi_{N_{i+1}}(a_t)}_t$ for  $\faktor{H_{i}}{H_{i+1}}$ that there exists a $b \in H_{i}$ such that $\|b\|_S \leq C_3^2 \pr{\|H\|_S}^{2 \:k_3}$ and where $\pi_{N_{i+1}}(b)$ generates $\faktor{H_{i}}{N_{i+1}}$. Take $S^\prime$ a generating subset for $N_{i}$ such that one generator projects to a generator for $\faktor{N_i}{N_{i+1}}$ and the others generate $N_{i+1}$. There exist constants $C_4, k_4 > 0$ such that $\Vert x \Vert_{S^\prime} \leq C_4 (\Vert x \Vert_S)^{k_4}$ for all $x \in N_{i}$ and so $\|b\|_{S^\prime} \leq C_3^{2 \: k_4} C_4 \pr{\|H\|_S}^{2 \:k_3 \: k_4}$. In particular $|d| \leq C_3^{2 \: k_4} \: C_4 \pr{\|H\|_S}^{2 \:k_3 \: k_4}$. Therefore,
	$$
	p^{\ell + \ell_0} \leq C_5 (\|H\|_S)^{m_2 + 2 \:k_3 \: k_4} \log(C_5\: \|x\|_S)
	$$
	for some constant $C_5 > 0$ as desired.
\end{proof}

\section{Proof of the main results}

\paragraph{Effective Separability of Subgroups}$
$\newline We first restate the first result.

\begin{repthm}{effective_separation_subgroup_thm}
	Let $N$ be a torsion-free, finitely generated nilpotent group with a finite generating subset $S$, and suppose $H \leq N$ is a subgroup of infinite index. Then for $h(N)$ the Hirsch length of $N$, it holds that $$\log(n) \preceq \Farb_{N,H,S}(n) \preceq \pr{\log(n)}^{h(N)}.$$ Moreover, if $H \nsub N$, there exists a constant $\psi(N,H) \in \N$ such that $$\log(n) \preceq \Farb_{N,H,S}(n) \preceq \pr{\log(n)}^{\psi(N,H)}.$$
\end{repthm}

We are now ready to give the proof of the first statement of the above result.
 \begin{proof}[Proof of first part of Theorem \ref{effective_separation_subgroup_thm}] We proceed by induction on Hirsch length of $N$, and since the base case is clear, we may assume that $h(N) > 1$. Let $\set{N_i}_{i=0}^{h(N)}$ be a maximal central series of $N$, and let $g \in N$ satisfy $g \notin H$ and $\|g\|_S \leq n$.
	
	 First suppose that $\pi_{N_{h(N)}}(g) \notin \pi_{N_{h(N)}}(H)$. Induction implies that there exists a surjective group morphism $\map{\pi}{\faktor{N}{N_{h(N)}}}{Q}$ to a finite group $Q$ such that $\pi(g) \notin \pi(\pi_{N_{h(N)}}(H))$. Moreover, there exists a constant $C_1 > 0$ such that 
	$
	|Q| \leq C_1 \pr{\log (C_1 n)}^{h(N)-1}.
	$ By composing $\pi$ with the natural projection on $\faktor{N}{N_{h(N)}}$, we are able to separate $g$ and $H$.
	
	Thus, we may assume that $g \in H N_{h(N)}$. Proposition \ref{length_of_g_mod_H} implies that there exists a $z \in N_{h(N)} \setminus H$ such that $g = h \: z$ where $h \in H$, and moreover, since $H$ is fixed, there exists a constant $C_2> 0$ and a $k_2 \in \N$ such that $\|z\|_S \leq C_2 \pr{\|g\|_S}^{k_2}$. Proposition \ref{central_element_sep_subgroup} implies that there exists a prime power $p^m$ such that $\pi_{p^m}(z) \notin \pi_{p^m}(H)$, and because $H$ is fixed, we get $
	p^m \leq C_3 \: \log(C_3 \:\|z\|_S )$
	for some constant $C_3 > 0$. We claim that $\pi_{p^m}(g) \notin \pi_{p^m}(H)$. For a contradiction, suppose otherwise. We then have $$\pi_{p^m}(g) = \pi_{p^m}(h \: z) \in \pi_{p^m}(H).$$ Since $h\in H$, we must have that $\pi_{p^m}(z) \in \pi_{p^m}(H)$ which is a contradiction. Thus, $$\D_N(H,g) \leq C_4 \: \pr{\log(C_4 \: n)}^{h(N)}$$ for some constant $C_4 > 0$, and subsequently, $$\Farb_{N,H,S}(n) \preceq \pr{\log(n)}^{ h(N)}.$$
	
	For the lower bound of $\Farb_{H,N,S}(n)$, let $\pi = \pi_{\sqrt[N]{[N,N]}}$.  Since $H$ is not a finite index subgroup of $N$, $\pi(H)$ is not a finite index subgroup of $\faktor{N}{\sqrt[N]{[N,N]}}$. Thus, there exists an element $g \in N$ such that $\pi(g^k) \notin \pi(H)$ for all $k > 0$. By changing the generating set of $N$ if necessary, we may assume that $\Vert g^k \Vert = k$. Letting $\set{p_i}$ be an enumeration of the primes, we let $k_i = \lcm\set{p_1,\cdots, p_{i -1}}$, and we define $g_i =g^{k_i}$. We then have that \cite[Cor 10.1.]{Tenenbaum} implies that $\log(k_i) \approx p_i$, and hence, we also have that $$ \log\left(\Vert \pi(g_i) \Vert_{\pi(S)}\right) = \log \left( k_i \right) \approx p_i.$$ It suffices to show that if $\map{\pi_2}{N}{Q}$ is a surjective group morphism to a finite group where $|Q| < p_i$, then $\pi_2(g_i) \in \pi_2(H)$. Since $|Q| < p_i$, it follows $\Ord_Q(\pi_2(g)) < p_i$. Thus, $\Ord_Q(\pi_2(g)) \mid k_i$ and hence, $\pi_2(g_i) = \pi_2(g^{k_i}) = 1$. Subsequently, $\pi_2(g_i) \in \pi_2(H)$, and therefore, $\D_N(H,g_i) \geq p_i$. Hence, $$\log(n) \preceq \Farb_{N,H,S}(n) \qedhere.$$
\end{proof}

For the second part of Theorem \ref{effective_separation_subgroup_thm}, we have to relate the complexity of separating $H$ in $N$ is to the complexity of residual finiteness in the quotient group $\faktor{N}{H}$ for a normal subgroup $H$ of $N$. The following lemma gives such a result for general finitely generated groups $G$.

\begin{lemma}\label{subsep_rf_quotient}
	Let $G$ be a finitely generated group with a normal finitely generated subgroup $H \nsub G$. The subgroup $H$ is a separable subgroup of $G$ if and only if the group $\faktor{G}{H}$ is residually finite. Moreover, if $S$ and $T$ are finite generating subsets of $G$ and $\faktor{G}{H}$, respectively, then $$\Farb_{G,H,S}(n) \approx \Farb_{\faktor{G}{H},\{1\},T}(n).$$
\end{lemma}

\begin{proof}For every $g \in G$, we will denote by $\pi_H\pr{g}$ as the natural projection of $g$ in $\faktor{G}{H}$. We demonstrate that $\D_{\faktor{G}{H}} \left(\set{1},\pi_H(g)\right) = \D_G(H,g)$ for all $g \notin H$, which implies the first statement. Note that if $\map{\pi}{\faktor{G}{H}}{Q}$ is a morphism such that $\pi(\pi_H(g)) \neq 1$, then $\pi \circ \pi_H : G \to Q$ is by definition a morphism which separates $H$ and $g$. Thus, $$\D_G(H,g) \leq \D_{\faktor{G}{H}} \left(\set{1},\pi_H(g)\right).$$
	
For the other inequality, assume that $\map{\pi}{G}{Q}$ is a surjective group morphism such that $|Q| =\D_G(H,g)$ and where $\pi(g) \notin \pi(H)$. Since $\pi(H)$ is a normal subgroup in $Q$, we can consider the quotient $\faktor{G}{\pi(H)}$, and by assumption, we have that $\pi_{\pi(H)}(\pi(g)) \neq 1$. 
By considering the induced group morphism $\tilde{\pi}: \faktor{G}{H} \to \faktor{Q}{\pi(H)}$, we note that this group morphism separates $\{1\}$ and $\pi_H(g)$ by construction. Therefore, $$\D_{\faktor{G}{H}}\left(\set{1},\pi_H(g)\right) \leq \D_G(H,g)$$ which gives our claim. 

The final statement of the lemma does not depend on the finite generating subsets, so we take $S$ any symmetric generating subset of $G$ and $T= \pi_H(S)$ to be a generating subset of $\faktor{G}{H}$. Suppose $g \in \faktor{G}{H}$ such that $\|g\|_S \leq n$. Since $\|\pi_H(g)\|_{T} \leq n$, we have $$\D_G(H,g) = \D_{\faktor{G}{H}}(\set{1},\pi_H(g)) \leq \Farb_{\faktor{G}{H},\set{1},T}(n).$$ Thus, $$\Farb_{G,H,S}(n) \leq \Farb_{\faktor{G}{H},\set{1},T}(n).$$

Now suppose $\pi_H(g_n) \in \faktor{G}{H}$ such that $\pi_H(g_n) \neq 1$ and where $\D_{\faktor{G}{H}} \left(\set{1},\pi_H(g)\right) = \Farb_{\faktor{G}{H},\set{1},T}(n)$. We may write
	$\pi_H(g_n) = \prod_{i=1}^{n}\pi_H(s_{i})$ where $s_i \in S$. That implies if we set $\tilde{g}_n = \prod_{i=1}^n s_i$, then $\|\tilde{g}_n\|_S \leq n$ and $\pi_H(\tilde{g}_n) = \pi_H(g_n)$. By the above claim, $\D_{G}(H,\tilde{g}_n) = \D_{\faktor{G}{H}}(\set{1},\pi_H(g_n))$, and thus, $\D_{G}(H,\tilde{g}_n) \leq \Farb_{G,H,S}(n)$. Hence, $$\Farb_{\faktor{G}{H},\set{1},S}(n) \leq \Farb_{G,H,S}(n),$$ and subsequently, $$\Farb_{G,H,S}(n) \approx \Farb_{\faktor{G}{H},\set{1},T}(n). \qedhere$$
\end{proof}

We now finish this subsection with the proof of the second statement of Theorem \ref{effective_separation_subgroup_thm}.
 \begin{proof}[Proof of second part of Theorem \ref{effective_separation_subgroup_thm}]
The statement is immediate, since if $H$ is a normal subgroup, Proposition \ref{subsep_rf_quotient} implies that $\Farb_{N,H,S}(n) \approx \Farb_{\faktor{N}{H},\set{1},T}(n)$ with $T$ any generating subset for $\faktor{N}{H}$. Since $\faktor{N}{H}$ is infinite, both \cite[Theorem 2.2]{Bou_Rabee10} and \cite[Theorem 1.3]{corrigendum_pengitore} imply the statement. \end{proof}

\paragraph{Effective Subgroup Separability}
Next we give the proof of our second main result. 
\begin{repthm}{effective_subgroup_separability_thm}
	Let $N$ be a torsion-free, finitely generated nilpotent group with a finite generating subset $S$. There exists a $k \in \N$ such that $$n \preceq \Sub_{N,S}(n) \preceq n^{k}.$$
\end{repthm}

\begin{proof}
We start with the upper bound for $\Sub_{N,S}(n)$, which is similar to the proof of the upper bound in Theorem \ref{effective_separation_subgroup_thm}. Let $\set{N_i}_{i=0}^{h(N)}$ be a maximal central series of $N$. Let $H\leq N$ be a subgroup where $\|H\|_S \leq n$, and let $g \in N \setminus H$ such that $\|g\|_S \leq n$. If $\pi_{N_{h(N)}}(g) \notin \pi_{N_{h(N)}}(H)$, then by induction there exists a constant $C_1 > 0$ and $k_1 \in \N$ such that $\D_{N}(H,g) \leq C_1 \: n^{k_1}$. Otherwise, Proposition \ref{length_of_g_mod_H} implies that there exists a $z \in N_{h(N)} \setminus  H$ and $h \in H$  such that $g = z \: h$. Moreover, there exists some constant $C_2 > 0$ and $k_2 \in \N$ such that 
$$
\|z\|_S \leq C_2 \: \pr{\text{max}\set{\|H\|_S,\|g\|_S}}^{k_2} \leq C_2 \: n^{k_2}.
$$

Proposition \ref{central_element_sep_subgroup} implies that there exists a prime power $p^k$ such that $\pi_{p^m}(z) \notin \pi_{p^m}(H)$ and where $$
	p^m \leq C_3 \: \|H\|_S^{k_3} \log(C_2 \: \|z\|_S).
	$$ 
	for some constant $C_3 > 0$ and integer $k_3$. Hence also $\pi_{p^m}(g) \notin \pi_{p^m}(H)$, since otherwise $\pi_{p^m}(z)  = \pi_{p^m}(g \: h^{-1}) \in \pi_{p^m}(H)$. Thus, there exist constants $C_4, k_4 > 0$ such that $\D_{N}(H,g) \leq C_4 \: n^{ k_4}$. Subsequently,
	$$
	\Sub_{N,S}(n) \preceq n^{k_4}.
	$$
	
	We now construct the lower bound for $\Sub_{N,S}(n)$ when $N$ is a finitely generated nilpotent group. Letting $M = \sqrt[N]{[N,N]}$, there exists a $g \in N$ such that $\pi_M(g)$ is a primitive non-trivial element of $\faktor{N}{M}$. Let $H_i = \innp{g^{p_i}}$. We claim that $\D_N(H_i,g) \geq p_i$ and since \cite[3.B2]{Gromov} implies that $\|H_i\|_S \approx p_i$, this will give us the lower bound. 
Thus, we need demonstrate if $\map{\pi}{N}{Q}$ is a surjective group morphism where $|Q| < p_i$ that $\pi(g) \in \pi(H_i)$. Now let $\la = \Ord_Q(\pi(g)) < p_i$, then we have $\gcd(p_i,\la) = 1$. Thus, there exists a $t \in \Z$ such that $t \: p_i \equiv 1 \mod \la \Z$. Hence, $\pi(g^{t \: p_i}) = \pi(g)$, and subsequently, $\pi(g) \in \pi(H)$.
\end{proof}

\section{Finite Extensions}
Using the methods of \cite[\S 7]{dere_pengitore_17}, this section demonstrates that we may generalize the methods to virtually nilpotent groups. 
\begin{thm}
	\label{virtually}
	Let $G$ be a finitely generated group with finite generating subset $S$, and assume that $G$ has a finite index normal subgroup $N$ with a finite generating subset $S^\prime$. For any separable subgroup $H \le G$ of infinite index, we have that there exists a $k \in \N$ such that $$\Farb_{N,N \cap H,S^\prime}(n) \preceq  \Farb_{G,H,S} (n) \preceq \left( \Farb_{N, N\cap H, S^\prime} (n) \right)^{([G:N])^2}.$$ Moreover, for the subgroup separability function, it holds that $$ \Sub_{N,S^\prime}(n) \preceq \Sub_{G,S} (n) \preceq \left( \Sub_{N, S^\prime} (n) \right)^{([G:N])^2}.$$
\end{thm}

\begin{proof}
First take a fixed subgroup $H \le G$, and consider the subgroup $H_0 = H \cap N$. Take elements $h_i \in H$ such that $\cup_{i=1}^m H_0 \: h_i= H$. Separating an element $g \in G \setminus H$ is equivalent to separating $g$ from $k$ right translates of $H_0 \le N$ where $k \leq [G:N]$. A combination of \cite[Lemma 7.3.]{dere_pengitore_17}, \cite[Lemma 7.4.]{dere_pengitore_17} and \cite[Lemma 7.5.]{dere_pengitore_17} now gives the upper half of the first inequality. 

For the lower bound of the first inequality, we let $g \in N \setminus N \cap H$ such that $\|g\|_{S^\prime} \leq n$. We note that if $g \in N \setminus N \cap H$, then also $g \in G \setminus H$. There exists a constant $C> 0$ independent of $n$ such that $\|g\|_{S} \leq C n$, since $N$ is a finite index subgroup.  Hence, there exists a surjective group morphism $\map{\pi}{G}{Q}$ such that $\pi(g) \notin \pi(H)$ and $|Q| \leq \Farb_{G,H,S}(C \: n)$. Since $\pi(g) \notin \pi(H)$, it follows that $\pi(g) \notin \pi(N \cap H)$, and thus, $\D_N(N \cap H,g) \leq \Farb_{G,H,S}(C\: n)$. Subsequently, $$\Farb_{N,N \cap H,S^\prime}(n) \preceq \Farb_{G, H,S}(n).$$

For the upper bound of the second inequality, it suffices to show that there exists a constant $C > 0$ such that for every finitely generated subgroup $H \le G$, it holds that $\Vert H_0 \Vert_S \leq C \Vert H \Vert.$ Indeed, in this case we can use the same methods as before to find the conclusion. To see that such a constant $C$ exists, let $H$ be any finitely generated subgroup of $G$ and fix generators $t_i \in H$ with $\Vert t_i \Vert_S \leq \Vert H \Vert_S$. Take elements $h_i \in H$ such that $$H = \bigcup_{i=1}^m h_i \: H_0 = \bigcup_{i=1}^k H_0 \: h_i.$$ Given that the diameter of the group $\faktor{H}{H_0}$ is bounded above by $[G:N]$, we may assume that $\Vert h_i \Vert_S \leq [G:N] \Vert t_i \Vert_S$. Schreier's Lemma implies that a finite  generating subset for $H_0$ is given by the elements $h_i \: t_j \: h_{i^\prime}^{-1}$ with $h_j \in T$ which lie in $H$. We conclude that $\Vert H_0 \Vert \leq C \Vert H \Vert$ with $C = 2 [G:N] + 1$.

For the lower bound of the second inequality, let $H \leq N$ and $g \in N \setminus N \cap H$ such that $\|H\|_{S^\prime},\|g\|_{S^\prime} \leq n$. As before, we note that $\|H\|_{S},\|g\|_S \leq C \: n$ for some $C > 0$ independent of $n$. Hence, there exists a group morphism $\map{\pi}{G}{Q}$ such that $\pi(g) \notin \pi(H)$ and where $|Q| \leq \Sub_{G,S}(C \: n)$. That implies $\D_N(N \cap H,g) \leq \Sub_{G,S}(C \: n)$. Thus, $$\Sub_{N,S^\prime}(n) \preceq \Sub_{G,S}(n). \qedhere $$
\end{proof}

\begin{cor}
Let $\Gamma$ be an infinite, finitely generated group with finite generating subset $S$, and let $N$ be a $\mathcal{F}$-group that is isomorphic to a normal finite index subgroup of $\Gamma$.  For every infinite index, finitely generated subgroup $H \le \Gamma$, we have that $$\log(n) \preceq \Farb_{\Gamma,H,S} (n) \preceq \left( \log(n) \right)^{h(N) \: \pr{[\Gamma:N]}^2}.$$ Moreover, if $k$ is the natural number from Theorem \ref{effective_subgroup_separability_thm} for $N$, then $$n \preceq \Sub_{\Gamma,S} (n)\preceq n^{k \: \pr{[\Gamma:N]}^2}.$$ 
\end{cor}

\begin{proof}
This follows immediately from Theorem \ref{virtually}, Theorem \ref{effective_separation_subgroup_thm} and Theorem \ref{effective_subgroup_separability_thm}.
\end{proof}

\section{Open questions}
\label{open}
In this section, we state some open questions about the exact form of the functions we introduced before. In Theorem \ref{effective_separation_subgroup_thm} and Theorem \ref{effective_subgroup_separability_thm} we gave the first lower and upper bounds for the functions $\Sub_{N,S}(n)$ and $\Farb_{N,H,S}(n)$, but the realized bounds are far from sharp. The techniques we develop do not allow to give sharp bounds, but we give some idea here how to improve the bounds.

\paragraph{Separability for subgroups}

One of our main results gives lower and upper bounds for $\Farb_{N,H,S}(n)$ and we believe that the exact bound lies inbetween.

\begin{ques}
Let $N$ be a $\mathcal{F}$-group with generating set $S$ and $H \le N$ a subgroup. Does there exist some $0 \leq k \leq h(N)$ such that $$\Farb_{N,H,S}(n) \approx \pr{\log(n)}^{k}?$$
\end{ques}

As an example, we show that it holds for the Heisenberg group  $H_3(\Z)$, which is given by the presentation $$H_3(\Z) = \langle a, b, c \mid [a,b] = c, [a,c] = [b,c] = 1 \rangle.$$ Take $S = \{a, b, c \}$ the standard generating set. Note that the center of this group $Z(H_3(\Z))$ is generated by $c$. If a subgroup $H \le H_3(\Z)$ has finite index, then the function $\Farb_{H_3(\Z),H,S}(n)$ is bounded. Thus, from now on we assume that $H$ has infinite index. We distinguish between two types of subgroups, depending on whether $H \cap Z(H_3(\Z))$ is trivial or not.

\begin{prop}
	Let $H \leq \text{H}_3(\Z)$ a subgroup of infinite index such that $H \cap Z(H_3(\Z)) \neq \set{1}$. Then $$\Farb_{ \text{H}_3(\Z),H,S}(n) \approx \log(n).$$
\end{prop}
\begin{proof}
	Let $H \leq \text{H}_3(\Z)$ be an infinite index subgroup and fix a generator $c^k$ with $k \neq 0$ for $H \cap Z(H_3(\Z))$. Theorem \ref{effective_separation_subgroup_thm} implies that $\log(n) \preceq \Farb_{ \text{H}_3(\Z),H,S}(n)$. Therefore, we need only to give a sharp upper bound.
	
	Let $x \in \text{H}_3(\Z) \backslash H$ such that $\|x\|_S \leq n$. In the case where $\pi_{\text{ab}}(x) \notin \pi_{\text{ab}}(H)$, we can write $B = \pi_{ \text{ab}}(H), \hspace{1mm} A = \pi_{ \text{ab}}(H_3(\Z))$ and use Corollary \ref{abeliancor} on the group $\faktor{A}{B}$.  
	
	In the other case, we have that $x = h \: c^l$ with $h \in H$ and $c^l \notin \langle c^k \rangle$. We can always assume that $0 < l < k$ and hence there are only finitely many possibilities. Proposition \ref{central_element_sep_subgroup} implies that there exists a $D \in \N$, independent of $l$, such that every $c^l$ with $0 < l < k$ can be separated from $H$ in a quotient of order $\leq D$. Just as in the proof of Theorem \ref{effective_separation_subgroup_thm}, we find that $x$ is also separated from $H$ in this quotient. Thus, in the second case, we get an upper bound independent of $\Vert x \Vert_S$.
\end{proof}
\begin{prop}
	Let $H \leq \text{H}_3(\Z)$ such that $H \cap Z(H_3(\Z)) = \set{1}$. Then $$\Farb_{\text{H}_3(\Z),H,S}(n) \approx \pr{\log(n)}^3.$$
\end{prop}
\begin{proof}
	The upper bound follows from Theorem \ref{effective_separation_subgroup_thm}, so it suffices to give the lower bound. For this we refer to the proof of \cite[Proposition 5.2.]{Pengitore_1}, which gives a sequence of central elements $x_i$ in $H_3(\Z)$ with $D_{H_3(\Z)}(\set{1},x_i) \approx \left( \log \left( \Vert x_i \Vert_S \right) \right)^3$ and $\Vert x_i \Vert_S \to \infty$. Since the elements are central, we have $x_i \notin H$, and moreover $$D_{H_3(\Z)}(\set{1},x_i) \leq D_{H_3(\Z))}(H,x_i),$$ so this gives us the lower bound.
\end{proof}

We conclude that for every subgroup $H \le H_3(\Z)$, we have $\Farb_{H_3(\Z),H,S}(n) \approx  \left( \log(n) \right)^k$ with $k \in \{0, 1, 3 \}$, and additionally, we have that there is an easy way to decide the correct value of $k$.

\paragraph{Subgroup separability}

For the subgroup function, we start by looking at some abelian examples.

\begin{Ex}
Consider the group $\Z^2$ with standard generating set $S = \left\{ e_1, e_2 \right\}$ and take any prime $p \in \Z$. Let $H$ be the subgroup generated by the elements $(1,p)$ and $(p,0)$. Note that the subgroup $H \cap \langle e_2 \rangle$ is generated by $(0,p^2)$ and hence has norm $p^2$. The original subgroup $H$ had norm $p+1$. By choosing increasing primes $p$, one thus finds an example for which $$\Vert H \cap \langle e_2 \rangle \Vert_S \approx \Vert H \Vert^2.$$ 

To separate the element $(0,p)$ from the subgroup $H$, we need a quotient of order at least $p^2$. This shows that $n^2 \preceq \Sub_{\Z^2,S}(n)$. To see that in fact $\Sub_{\Z^2,S}(n) \approx n^2$, we can use \cite[Proposition 4.2.]{dere_pengitore_17}. Similarly one can show that $\Sub_{\Z^k,S}(n) \approx n^k$, but we leave the details to the reader to check. In particular, for every $d \in \N$, there exists a finitely generated group $G$ such that $\Sub_{G,S}(n) \approx n^d$. 
\end{Ex}

It is an open problem to understand the asymptotic behaviour of the function $\Sub_{N,S}(n)$ for nilpotent groups $N$.  
\begin{ques}
	Let $N$ be a $\mathcal{F}$-group with generating set $S$. Is is true that $\Sub_{N,S}(n) \approx n^k$ for some $k \in \N$?
\end{ques}

We can do a similar construction in the Heisenberg group, showing that the degree of the subgroup fuctions can grow more rapidly than the Hirsch length.

\begin{Ex}
Let $H \le H_3(\Z)$ be the subgroup generated by the elements $b \: c^{p^2}$ and $b^p$. Just as in the previous example we get that $H \cap Z(H_3(\Z)) = \langle c^{p^3} \rangle $ and hence $\Vert H \cap Z(H_3(\Z)) \Vert_S \approx p^{\frac{3}{2}}$ whereas $\Vert H \Vert_S \approx p$. To separate the element $c^{p^2}$ with $\Vert c^{p^2} \Vert_S \approx p$ from this group, we need a quotient of order at least $p^9$. This implies that $n^9 \preceq \Sub_{H_3(\Z),S}(n)$.
\end{Ex}

If we would estimate the degree of the upper bound in Theorem \ref{effective_subgroup_separability_thm}, then it would only depend on $h(N)$, but it would grow exponentially with the Hirsch length. 

\bibliography{bib}

\begin{thebibliography}{10}

\bibitem{Agol_virtual_haken}
Ian Agol.
\newblock The virtual haken conjecture.
\newblock {\em Doc. Math.}, 18:1045--1087, 2013.

\bibitem{Bou_Rabee10}
Khalid Bou-Rabee.
\newblock Quantifying residual finiteness.
\newblock {\em J. Algebra}, 323(3):729--737, 2010.

\bibitem{dere_pengitore_17}
Jonas D\'{e}re and Mark Pengitore.
\newblock Effective twisted conjugacy separability of nilpotent groups.
\newblock {\em Math. Z.}, 2018.

\bibitem{Gromov}
M.~Gromov.
\newblock Asymptotic invariants of infinite groups.
\newblock In {\em Geometric group theory, {V}ol.\ 2 ({S}ussex, 1991)}, volume
  182 of {\em London Math. Soc. Lecture Note Ser.}, pages 1--295. Cambridge
  Univ. Press, Cambridge, 1993.

\bibitem{Patel3}
Mark~F. Hagen and Priyam Patel.
\newblock Quantifying separability in virtually special groups.
\newblock {\em Pacific J. Math.}, 284(1):103--120, 2016.

\bibitem{Hall_notes}
Philip Hall.
\newblock {\em The {E}dmonton notes on nilpotent groups}.
\newblock Queen Mary College Mathematics Notes. Mathematics Department, Queen
  Mary College, London, 1969.

\bibitem{Hall_coset}
Marshall Hall~Jr.
\newblock Coset representations in free groups.
\newblock {\em Trans. Amer. Math. Soc.}, 67:421--432, 1949.

\bibitem{Eick_Holt_obrien}
Derek~F. Holt, Bettina Eick, and Eamonn~A. O'Brien.
\newblock {\em Handbook of computational group theory.}
\newblock Discrete Mathematics and its Applications. Chapman \& Hall/CRC, Boca
  Raton, FL., 2005.

\bibitem{LLM}
Sean Lawton, Larsen Louder, and D.~B. McReynolds.
\newblock Decision problems, complexity, traces, and representations.
\newblock {\em Groups Geom. Dyn.}, 11(1):165--188, 2017.

\bibitem{zariski_closure_subgroup_separability}
Larsen Louder, D.B. McReynolds, and Priyam Patel.
\newblock Zariski closures and subgroup separability.
\newblock {\em Selecta Math.}, 23(3):2019--2027, 2017.

\bibitem{optimal_Bezout_coefficients_bound}
Bohdan~S. Majewski and George Havas.
\newblock {\em The complexity of greatest common divisor computations}, pages
  184--193.
\newblock Lecture Notes in Comput. Sci. Springer Berlin Heidelberg, 1994.

\bibitem{Mal01}
A.I. Mal'tsev.
\newblock On homomorphisms onto finite groups.
\newblock {\em Ivanov. Gos. Ped. Inst. Ucen. Zap.}

\bibitem{Patel}
Priyam Patel.
\newblock On a theorem of {P}eter {S}cott.
\newblock {\em Proc. Amer. Math. Soc.}, 142(8):2891--2906, 2014.

\bibitem{Pengitore_1}
Mark Pengitore.
\newblock Effective separability of finitely generated nilpotent groups.
\newblock {\em New York J. Math.}, 24:83--145, 2018.

\bibitem{corrigendum_pengitore}
Mark Pengitore.
\newblock Residual dimension of nilpotent groups: correction to ``effective
  separability of finitely generated nilpotent groups”.
\newblock {\em Preprint}, arXiv:1809.07934.

\bibitem{scott_surface_subgroups}
Peter Scott.
\newblock Subgroups of surfacce groups are almost geometric.
\newblock {\em J. London Math. Soc. (2)}, 17(3):555--565.

\bibitem{Segal_book_polycyclic}
Daniel Segal.
\newblock {\em Polycyclic groups.}, volume~82 of {\em Cambridge Tracts in
  Mathematics}.
\newblock Cambridge University Press, Cambridge, 1983.

\bibitem{Tenenbaum}
G\'{e}rald Tenenbaum.
\newblock {\em Introduction to analytic and probabilistic number theory.},
  volume~46 of {\em Cambridge Studies in Advanced Mathematics.}
\newblock Cambridge University Press, Cambridge, 1995.

\bibitem{Wise}
D.~Wise.
\newblock Research announcement: the structure of groups with a quasiconvex
  hierarchy.
\newblock {\em Electron. Res. Announc. Math. Sci.}, 16:44--55, 2009.

\end{thebibliography}
\bibliographystyle{plain}
\end{document}